\documentclass[11pt,twoside]{article}
\usepackage[utf8]{inputenc}
\usepackage[T1]{fontenc}
\usepackage{latexsym}
\usepackage{amssymb,amsbsy,amsmath,amsfonts,amssymb,amscd}
\usepackage{cases}
\usepackage{graphicx,color}
\usepackage{float,url}
\usepackage{cancel}
\usepackage[colorlinks=true]{hyperref}
\usepackage{comment}

\newcommand{\commentout}[1]{}
\newcommand{\R}{\mathbb{R}}

\newcommand {\vp} {\varphi}

\newcommand {\Chi} {{\bf \raise 2pt \hbox{$\chi$}} }
\newcommand {\dt}   {\Delta t}

\newcommand {\cae} { {\mathcal E} }

\newcommand {\f}   {\frac}
\newcommand {\p}   {\partial}

\newcommand {\ov}  {\overline}
\newcommand{\beq}{\begin{equation}}
\newcommand{\eeq}{\end{equation}}
\newcommand{\bea} {\begin{array}{rl}}
\newcommand{\eea} {\end{array}}
\newcommand{\bepa}{\left\{ \begin{array}{l}}
\newcommand{\eepa} {\end{array}\right.}

\newcommand{\norm}[1]{\left\lVert#1\right\rVert}
\newcommand{\abs}[1]{\left\lvert#1\right\rvert}
\newcommand{\parent}[1]{\left(#1\right)}
\newcommand{\fespace}{V_h}
\newcommand*{\dd}{\mathop{\kern0pt\mathrm{d}}\!{}}

\newcommand{\scal}[2]{\left(#1 , #2\right)}
\newcommand{\lscal}[2]{\left(#1 , #2\right)}

\newcommand{\bas}{\phi}

\newcommand\restr[2]{{
  \left.\kern-\nulldelimiterspace 
  #1 
  \vphantom{\big|} 
  \right|_{#2} 
  }}

\renewcommand{\u}{u}
\newcommand{\nn}{n}

\newlength\eqnspace
\newtheorem{theorem}{Theorem}

\newtheorem{remark}[theorem]{Remark}
\newtheorem{proposition}[theorem]{Proposition}

\numberwithin{equation}{section}

\newcommand{\qed}{{ \hfill
                       {\unskip\kern 6pt\penalty 500 \raise -2pt\hbox{\vrule\vbox to 6pt{\hrule width 6pt
                       \vfill\hrule}\vrule} \par}   }}
\title{Scalar auxiliary variable finite element scheme for the parabolic-parabolic Keller-Segel model.}
\author{Alexandre Poulain\thanks{Sorbonne Universit\'{e}, CNRS,  Universit\'{e} de Paris, Inria, Laboratoire Jacques-Louis Lions, F-75005 Paris, France.} 
\thanks{Email: poulain@ljll.math.upmc.fr} 
 \thanks{The author has received funding from the European Research Council (ERC) under the European Union's Horizon 2020 research and innovation programme (grant agreement No 740623)}
}

\date{\today}

\begin{document}
\maketitle
\pagestyle{plain}
\pagenumbering{arabic}

\begin{abstract} 
We describe and analyze a finite element numerical scheme for the parabolic-parabolic Keller-Segel model. The scalar auxiliary variable method is used to retrieve the monotonic decay of the energy associated with the system at the discrete level. This method relies on the interpretation of the Keller-Segel model as a gradient flow. The resulting numerical scheme is efficient and easy to implement. 
We show the existence of a unique non-negative solution and that a modified discrete energy is obtained due to the use of the SAV method. We also prove the convergence of the discrete solutions to the ones of the weak form of the continuous Keller-Segel model.
\end{abstract} 
\vskip .7cm

\noindent{\makebox[1in]\hrulefill}\newline
2010 \textit{Mathematics Subject Classification.} 35K20 ; 35Q92; 65M12; 35K55 
\newline\textit{Keywords and phrases.} Keller-Segel; Living tissues; Gradient flow; Energy stability
%

\section{Introduction}

Since chemotaxis is observed very widely in various areas of biology and medicine, it becomes a prolific subject in mathematical biology throughout the past decades.  
Among the different mathematical models used to represent chemotaxis of living organisms, the Keller-Segel equation is one of the most recognized. 
It has been introduced by Keller and Segel \cite{keller_initiation_1970} to depict the movement of the Dictyostelium discoideum toward the location of high concentration of adenosine 3’, 5’-cyclic monophosphate. 
The parabolic-parabolic Keller-Segel model (KS in short) is often set in a bounded domain $\Omega \subset \R^d$, $d=1,2,3$ with a Lipschitz boundary $\p \Omega$ and reads
\begin{align}
    \p_t u = \nabla\cdot\left(D_u \nabla u - \chi_c \vp(u) \nabla c \right) &\quad \text{in} \quad  \Omega \times (0,+\infty), \label{eq:KS-1}\\
    \tau \p_t c = \Delta c - \alpha c + u &\quad \text{in} \quad  \Omega \times (0,+\infty), \label{eq:KS-2}
\end{align}
endowed with zero-flux boundary condition
\begin{equation}
    \f{\p \parent{D_u \nabla u - \chi_c \vp(u) \nabla c }}{\p \nu} = \f{\p c}{\p \nu} =0 \qquad \text{on} \quad  \p \Omega \times (0,+\infty),
    \label{eq:ks-boundary}
\end{equation}
where $\nu$ is the outward normal  vector to the boundary.
We assume in the following that the initial condition satisfies
\begin{equation}
    \{u(0,x),c(0,x)\} = \{u^0,c^0\} \in H^1(\Omega)\times H^1(\Omega),\quad \text{and} \quad 0 \le u^0 \le 1 \text{ a.e. in } \Omega, \quad  0 \le c^0 \le C \text{ a.e. in } \Omega,
    \label{eq:hyp-init-cond}
\end{equation}
where $C$ is a positive finite constant.
In the model \eqref{eq:KS-1}--\eqref{eq:hyp-init-cond}, the cell density $u(t,x)$ is attracted by the chemo-attractant given by $c(t,x)$, its concentration. 
Cells can move randomly by diffusion with a coefficient of diffusion given by $D_u$ and by chemotaxis with $\chi_c$, a coefficient used to represent the strength of this movement.
A small parameter $\tau>0$ is used to denote how fast the chemo-attractant is diffusing compared to the cells. Without a loss a generality, we will assume in the following that $\tau =1$. 

$\vp(u)$ is the chemosensitivity and is given by
\begin{equation}
    \vp(u) = u(1-u) \qquad \text{for} \qquad 0 \le u\le 1. 
    \label{eq:vp}
\end{equation}
This particular form of chemosensitivity prevents the unrealistic scenario of overcrowding of cells and therefore the blow-up of the solution. Due to this possible behavior of solutions, the Keller-segel system exhibits very interesting mathematical structure and the interested reader can refer to the review \cite{hillen_users_2008} and the work of Blanchet \textit{et al.} \cite{blanchet-keller-2006}. The volume filling strategy was proposed in \cite{painter_volume-filling_2002} to take into account the finite size of individual cells, leading to the form \eqref{eq:vp}.    

The Keller-Segel model ~\eqref{eq:KS-1}--\eqref{eq:KS-2} has a gradient flow structure with the associated energy
\begin{equation}
    \cae[u,c](t) = \int_\Omega B \left[u \log u - (u-1)\log(1-u)\right] -u c  +\f12 \left(\abs{\nabla c}^2 + \alpha c^2\right) + C  \, \dd x,
\end{equation}
where $B=D_u/\chi_c$ and we define the integral of the free energy density
\[
    \cae_1 [u](t) = B \int_\Omega  F(u)  \, \dd x,
\]
where 
\begin{equation}
    F(u) = u \log u - (u-1)\log(1-u) + C.
    \label{eq:free-energy}
\end{equation}
Here, $C$ is a positive constant such that $F(u) > 0$, $\forall u \in [0,1]$.
For latter convenience, we denote $F'(u) = g(u)$, and we remark that $g'(u) = \f{1}{\vp(u)}$.
Thus, we can express the Keller-Segel model using its gradient flow structure \cite{blanchet-ks-gradient-2013}
\begin{align}
    \p_t u &= \nabla\cdot\left( \chi_c \vp(u) \nabla \f{\delta \cae}{\delta u}\right), \label{eq:KS-1-gradient-flow}\\
    \tau \p_t c &= -\f{\delta \cae}{\delta c}, \label{eq:KS-2-gradient-flow}
\end{align}
where the variational derivatives of the energy functional with respect to $u$ and $c$ are given respectively by 
\begin{align*}
    \f{\delta \cae}{\delta u}  &=  B g(u) - c  ,\\
    \f{\delta \cae}{\delta c} &= -\Delta c + \alpha c - u.
\end{align*}

Generally, a numerical scheme for gradient flow model is evaluated by several aspects: \textit{i)} its capacity to keep the energy dissipation; \textit{ii)} if it is convergent, and if error bounds can be established; \textit{iii)} its efficiency; \textit{iv)} its implementation simplicity. For a large a class of gradient flows, the Scalar Auxiliary Variable (SAV in short) \cite{shen_new_2019} has shown to meet all the previous points.  
Applying this method to the Keller-Segel model is only possible starting from its gradient flow formulation  ~\eqref{eq:KS-1-gradient-flow}--\eqref{eq:KS-2-gradient-flow} and gives what we call the SAV Keller-Segel model
\begin{align}
    \p_t u &= \nabla \cdot \left(D_u \vp(u) \nabla \mu_1\right), \label{eq:KS1-SAV}\\
    \mu_1 &= B \f{r}{\sqrt{\cae_1[u]}} v_1[u] - c, \label{eq:KS11-SAV}\\
    \tau \p_t c &= - \mu_2, \label{eq:KS2-SAV}\\
    \mu_2 &= -\Delta c + \alpha c - u\label{eq:KS22-SAV},
\end{align}
where 
\begin{equation}
    v_1[u] = \f{\p \cae_1}{\p u}, 
\label{eq:continuous-v_i}
\end{equation}
and we define the scalar unknown
\begin{equation}
    \f{\dd r}{\dd t} = \f{1}{2\sqrt{\cae_1[u]}}\int_\Omega v_1[u]\f{\p u}{\p t} \, \dd x. 
    \label{eq:continuous-r_i}
\end{equation}
In this article, we propose to study a finite element scheme to simulate the system ~\eqref{eq:KS1-SAV}--\eqref{eq:continuous-r_i} that is known to preserve the energy at the discrete level.

Throughout the past decades, the Keller-Segel model as been at the center of many pieces of research. The analytical properties of the Keller-Segel model without volume filling have been extensively studied. One of the most important result was to show that the solution of the model blows up in finite time if a certain constraint on the initial mass is not satisfied.
For the reader interested into the analytical results about this model without volume filling, we refer to the review paper \cite{suzuki-chemotaxis-2018}. The volume filling approach prevents this blow-up of the solution in finite time for any initial condition satisfying \eqref{eq:hyp-init-cond}. Moreover, it seems to be more biologically relevant since it takes into account the finite size of the cells. A more general form of the Keller-Segel model is 
\begin{equation}
    \begin{cases}
            &\p_t u - \nabla \cdot \left( D_u \beta(u) \nabla u -\chi_c u\mu(u) \nabla c \right)=0,\\
            &\p_t c - D_c \Delta c = \delta u - \alpha c, 

    \end{cases}
\end{equation}
where the random movement of the cells is given by $D_u \beta(u)$ (that can be non-linear) and the chemosensitivity is given by $\chi_c\mu(u)$. Particular assumptions on both $\beta(u)$ and $\mu(u)$ can be made to prevent the blow-up of solutions in finite time.
The introduction of the parabolic-parabolic KS with volume filling and quorum-sensing is presented in the work of Painter and Hillen \cite{painter_volume-filling_2002}. They described a discrete lattice model where the probability for cells to jump to a different location is dependent on the local density and on the concentration of the chemotactic agent. From this discrete model, they derived the continuous limit model and give the following conditions for $\beta(u)$ and $\mu(u)$ 
\[
    \beta(u)  := \psi(u) - u\psi'(u),\quad \mu(u) \equiv \psi(u), 
\]
where $\psi(u)$ is a monotonically decreasing function. Under the assumptions 
\[
    \psi(0)>0 , \quad \psi(u)>0\quad \text{for}\quad 0 < u< \ov u,\quad \text{and}\quad \psi(\ov u) =0,
\]
the same authors proved the global existence of classical solutions in \cite{hillen_2001_existence}. They also presented some numerical simulations where they were able to make observation of the behavior of the solution for longer times since the blow-up of the solution is prevented by the model. Many other variations of the Keller-Segel model have been proposed to take into account the effect of volume filling. For example, more recently, Bubba \textit{et al.} \cite{bubba_2019_chemotaxis} proposed to take
\[
    \psi(u) = \exp\left(-\f{u}{u_\text{max}}\right),
\] 
where $u_\text{max}$ represents the density at which cells are too overcrowded. \par
Numerical methods for the Keller-Segel model are numerous. Considering zero-flux boundary conditions, the conservation of the total mass of the cells, the non-negativity of the solution and the capacity to retrieve the energy at the discrete level are the key properties expected from a numerical scheme for this equation. 
For the parabolic-elliptic Keller-Segel equation where the equation for the chemo-attractant is given by 
\[
    -\Delta c = \delta u-\alpha c,  
\]
Saito and Suzuki proposed a conservative finite-difference scheme \cite{saito_2005_fd}. 
For the parabolic-parabolic version, Saito proposed and performed an error analysis for an upwind finite element scheme \cite{saito_2007_fe,saito_conservative_2009,saito_2012_error} using Baba and Tabata's method \cite{baba_1981_upwindfe}. The finite volume method has also been applied for this problem: we can cite the work of Filbet \cite{filbet_2006_fv} that deals with the classical Patlak-Keller-Segel model (without volume filling) and the work of Almeida \textit{et al.} \cite{almeida_energy_2019}. In the latter, the parabolic-elliptic model is used and the authors were able to prove the preservation of the important properties for two finite volume schemes. The difference between the two is that one uses the gradient flow structure of the model while the user uses an exponential rewriting inspired by the Scharfetter-Gummel discretization. 
The scheme we propose in this article follows the same idea. The Keller-Segel model has a gradient flow structure that can be useful for its numerical simulation. 

A recent numerical method to simulate gradient flows that ensures that the energy is preserved at the discrete level is the SAV method \cite{shen_2018_sav,shen_new_2019}. This method provides a robust framework to simulate gradient flows in an efficient way. In fact, the computation of the solution of any gradient flow model requires only the solving of two decoupled linear systems at each time step. This method has shown very interesting results for the simulation of the Cahn-Hilliard equation \cite{shen_2018_convergence} for which the properties concerning the discrete energy and the conservation of the total mass are of main importance. We must stress that the energy recovered by the SAV method is a modified version of the energy of the real system. This is due to the discretization of the equation for the scalar variable. In a recent work of Bouchriti \textit{et al.} \cite{pierre-sav}, the authors showed that the use of the SAV method for the damped wave equation and the Cahn-Hilliard equation leads to the convergence to modified steady states as well. 

In a recent work Shen and Xu \cite{Shen_2020_unconditionally} proposed an unconditionally energy stable method that is able to preserve the nonnegativity of the solution. This method relies on the use of the gradient flow structure of the Keller-Segel model and requires the solving of a nonlinear convex system.  

To the best of our knowledge the SAV method has never been applied to the Keller-Segel model. The principal difference with previous works on the SAV method is that the mobility in the first equation of the Keller-Segel system is not constant through time, leading to the necessity to compute at each time step the associated matrix. 

Therefore, in this article, we propose to use it to obtain a new model that we discretize in space using the finite element method. Altogether, we obtain a new way to simulate the parabolic-parabolic Keller-Segel equation with the certitude to be able to retrieve the energy associated with the model at the discrete level. First, we describe the method and explain the strategy to solve the resulting equations. Then, the well-posedness of the scheme is studied. We show the existence of a unique pair of solution that is non-negative and retrieve the expected $L^\infty$ norm under some constraints on the spatio-temporal mesh. We also show that the initial mass of the cells is conserved. We prove that a modified energy is retrieved at the discrete level which is an inherent property of the SAV method. Lastly, a convergence analysis is performed. The convergence of subsequences in Bochner spaces can be proved and the system solved by the limit solution is the weak form of the Keller-Segel equation.  

\section{Numerical scheme}
\subsection{Finite element framework}
Let $L^p(\Omega)$, $W^{m,p}(\Omega)$ with $H^m(\Omega) = W^{m,2}(\Omega)$, where $1 \le p \le +\infty$ and $m \in  \mathbb{N}$, be respectively the usual Lebesgue and Sobolev spaces. The corresponding norms are respectively $||\cdot||_{m,p,\Omega}$, $||\cdot||_{m,\Omega}$ and semi-norms $|\cdot|_{m,p,\Omega}$, $|\cdot|_{m,\Omega}$. 
We denote $L^p\left(0,T;V\right)$ the Bochner spaces i.e. the spaces with values in Sobolev spaces \cite{adams-sobolev-1975}. The norm in these spaces is defined for all function $\eta$ Bochner measurable by
\[
    \norm{\eta}_{L^p(0,T;V)} = \left(\int_0^T \norm{\eta}^p_V \, \dd t\right)^{1/p},
\] 
and
\[
    \norm{\eta}_{L^\infty(0,T;V)} = \text{ess} \sup_{t\in (0,T)} \norm{\eta}_V. 
\]
The standard $L^2$ inner product is denoted by $(\cdot,\cdot)_\Omega$ and the duality pairing between $(H^1(\Omega))'$ and $H^1(\Omega)$ by $<\cdot,\cdot>_\Omega$.

Let $\Omega$ be a polyhedral domain and $\mathcal{T}^h$, $h>0$, be a quasi-uniform mesh of this domain into $\abs{\mathcal{T}^h}$ disjoint open mesh elements $T$. Let $h_T := \text{diam}(T)$ and $h = \max_{T\in \mathcal{T}^h} h_T$. Since the mesh is assumed to be quasi-uniform, we know that it is shape-regular and it exists a positive constant $C$ such that 
\begin{equation*}
    h_T \ge C h, \quad \forall T \in \mathcal{T}^h.
    \label{eq:quasi-uniform}
\end{equation*} 
Since the domain is assumed to be polyhedral, the discrete domain $\Omega_h$ exactly coïncides with the domain $\Omega$.    
Hence, the closure of the domain can be written as the union of all the mesh elements $\overline{\Omega} = \ov \Omega_h = \bigcup_{T\in T^h} \overline{T}$. We assume that the mesh is acute, \textit{i.e.} for $d=2$ the angles of the triangles can not exceed $\frac{\pi}{2}$ and for $d=3$ the angle between two faces of the same tetrahedron can not exceed $\frac{\pi}{2}$. We define by $\kappa_T$ the minimal perpendicular length of $T$ and $\kappa_h = \min_{T\in \mathcal{T}^h} \kappa_T$.
 We introduce the P-1 finite element space associated with the mesh $\mathcal{T}^h$
\begin{equation*}
    \fespace := \{ \bas \in C(\overline{\Omega}):  \restr{\bas}T \in \mathbb{P}^1(T), \quad \forall T \in \mathcal{T}^h \} \subset H^1(\Omega),
\end{equation*}
where $\mathbb{P}^1(T)$ denotes the space of polynomials of order $1$ on $T$.
For latter convenience, we indicate the set of nodes of $\mathcal{T}^h$ by $J_h$ and $\{x_j\}_{j=1,\dots,\abs{J_h}}$ is the set of their coordinates. $N_h = \abs{J_h}$ stands for the total number of nodes. We denote by $\Lambda_i$ the set of nodes connected to the node $x_i$ by an edge and $G_h = \max_{x_i\in J_h} \abs{\Lambda_i}$.  $\{ \bas_j\}_{j=1,\dots, N_h}$ is the standard Lagrangian basis functions associated with the spatial mesh.

The standard interpolation operator is defined by $\pi^h:C(\ov\Omega) \to \fespace$ such that $\pi^h(\eta(x_j)) = \eta(x_j)$ for all $x_j\in J_h$.
We also define the $L^2$ projection operator $P_h: L^2(\Omega)\to \fespace$ 
\[
    \begin{aligned}
        \scal{P_h v}{\bas} &= \scal{v}{\bas} \quad \forall v\in L^2(\Omega)\text{ and } \forall \bas \in \fespace,
    \end{aligned}
\]

For latter convenience, we state here some well-known results for the P-1 finite element method (see for e.g. \cite{brenner_2008_fe}, \cite{quarteroni_numerical_1994})
\begin{align}
    &\abs{\chi}_{m,p_2} \le C h^{-d\left(\f{1}{p_1}-\f{1}{p_2} \right)}\abs{\chi}_{m,p_1} \quad \forall \chi \in S^h, 1\le p_1\le p_2 \le  +\infty, m=0,1; \label{eq:ineq-comparaison-norm} \\
    &\lim_{h\to 0}\norm{v-\pi^h(v)}_{0,\infty}=0 \quad \forall v \in C(\ov \Omega),
    \label{eq:ineq-lumped-scalar-product2}\\
    &\abs{v-P_h v}_{0} + h\abs{v-P_h v}_1 \le C h^m \norm{v}_{m} \quad v\in H^m(\Omega), \quad m=1,2.\label{eq:L2-error}
\end{align}

We define the standard mass $M$ and stiffness $K$ finite element matrices
\begin{equation*}
	M_{ij} = \int_{\Omega} \bas_i  \bas_j \, \text{d}x, \quad \text{ for } i,j = 1, \dots, N_h,
\end{equation*} 
\begin{equation*}
	K_{ij} = \int_{\Omega} \nabla \bas_i \nabla \bas_j \, \text{d}x, \quad \text{ for } i,j = 1, \dots, N_h.
\end{equation*} 
For the efficiency of the numerical scheme, it could be useful to use the lumped mass matrix which is a diagonal matrix with each term being the sum of the terms on the same row of the standard mass matrix.
\begin{equation*}
	M_{l,ii} := \sum_{j=1}^{N_h} M_{ij}	 \quad \text{ for } i,j = 1, \dots, N_h.
\end{equation*}
From the hypothesis we made on the acuteness of the triangulation, we know that (see \cite{fujii_1973_fe})
\[
    \scal{\nabla \bas_i}{\nabla \bas_j} \le 0, \quad \text{ for } i\neq j.
\]
Therefore, we know that the non-diagonal entries of the stiffness matrix $K$ and of the matrix $A$ defined below by the equation \eqref{eq:matrix-A} are non-positive.
\subsection{Fully discrete scheme}
Given $N_T \in \mathbb{N}^*$, let $\Delta t := T / N_T$ be the constant time-step and $t^n := n \Delta t$, for $n= 0, \dots, N_T - 1$. We consider a partitioning of the time interval $[0,T] = \bigcup_{n=0}^{N_T-1}[t^n,t^{n+1}]$.
We approximate the continuous time derivative using a forward Euler method $\f{\p u_h}{\p t} \approx \f{u^{n+1}_h - u^n_h}{\dt}$.
The finite element numerical problem associated with the system ~\eqref{eq:KS1-SAV}--\eqref{eq:continuous-r_i} is: 

Find $\{u^{n+1}_h,c^{n+1}_h\}\in \fespace\times \fespace$ such that $\forall \bas \in \fespace$
\begin{align}
    \lscal{\f{\u^{n+1}_h-\u^n_h}{\dt}}{\bas} &= -\chi_c \scal{\vp(\u^n_h) \nabla \mu^{n+1}_{1,h}}{\nabla \bas},  \label{eq:p_t u}\\
    \lscal{\f{c^{n+1}_h-c^n_h}{\dt}}{\bas} &= -\lscal{\mu^{n+1}_{2,h}}{\bas}, \label{eq:p_t c}\\
    \lscal{\mu^{n+1}_{1,h}}{\bas} &= -\lscal{c^{n}_h}{\bas} + B \lscal{\f{ P_h\left(v_{1,h}[u^{\nn}_h] \right)}{\sqrt{\cae_1[u^{\nn}_h]}}}{\bas}r^{n+1}, \label{eq:mu_1}\\
    \lscal{\mu^{n+1}_{2,h}}{\bas} &= \scal{\nabla c^{n+1}_h}{\nabla \bas} + \alpha\lscal{c^{n+1}_h}{\bas} - \lscal{u^{n+1}_h}{\bas}
    \label{eq:mu_2},\\
    r^{n+1}-r^n &= \f12 \lscal{ \f{P_h\left(v_{1,h}[ u^{\nn}_h]\right)}{ \sqrt{\cae_1[u^{\nn}_h]}}}{ (u^{n+1}_h-u^n_h)}, \label{eq:r_n_1}
\end{align}
where $u_h^n(x) = \sum_{j=1}^{N_h}u^n_j \bas_j(x)$ and $c^n_h(x) = \sum_{j=1}^{N_h}c^n_j \bas_j(x)$ are respectively the finite element approximations of the cell density $u$ and the concentration of the chemo-attractant $c$. We also have used the notation $v_{1,h}= \f{\p \cae_1[u^n_h]}{\p u^n_h} $. We add to this system the following initial conditions
\begin{equation}
    \begin{cases}
        & \{u_h^0,c^0_h\}=\{\pi^h u^0,\pi^h c^0\}\quad \text{if } d=1,\\
        & \{u_h^0,c^0_h\}=\{ P_h u^0, P_h c^0\} \quad \text{if } d=2,3.
    \end{cases}
    \label{eq:init-cond}
\end{equation}

\subsection{Matrix formulation}
Let us define $A$ the finite element matrix associated with the right-hand side of \eqref{eq:p_t u}
\begin{equation}
    A_{ij}^n = \int_\Omega  \vp(u^n_h) \nabla \bas_i \nabla \bas_j \, \dd x  \quad \text{ for } i,j = 1, \dots, N_h,
    \label{eq:matrix-A}
\end{equation}
and the variable
\begin{equation}
    s^{\nn}_{1,h} =   \f{ P_h\left(v_{1,h}[u^{\nn}_h] \right)}{\sqrt{\cae_1[u^{\nn}_h]}}.
    \label{eq:def-s_h}
\end{equation}
We denote in the following by capital letters the vectors associated with the quantities denoted by small letters in the finite element problem.
Therefore, the system~\eqref{eq:p_t u}--\eqref{eq:r_n_1} can be rewritten into a matrix formulation 
\begin{align}
    M \f{U^{n+1}-U^n}{\dt} &= - \chi_c A^n W^{n+1}_1, \label{eq:p_t u mat}\\
    M\f{C^{n+1}-C^n}{\dt} &= -M W^{n+1}_2, \label{eq:p_t c mat}\\
    M W^{n+1}_1 &= -M C^n  + B M  S_1^{\nn} r^{n+1} , \label{eq:mu_1 mat}\\
    M W^{n+1}_2 &= K C^{n+1} +  \alpha M C^{n+1} - M U^{n+1}. \label{eq:mu_2 mat}
\end{align}

\subsection{Linear system}
Replacing \eqref{eq:mu_1 mat} into \eqref{eq:p_t u mat} but also \eqref{eq:mu_2 mat} into \eqref{eq:p_t c mat}, we obtain the system
\begin{align}
    \frac{M}{\dt} U^{n+1} + \f12\scal{s_{1,h}^{\nn}}{u_h^{n+1}} A^n  S_1^{\nn} &= L^n_1, \label{eq:lin-1}\\
    -M U^{n+1} + \left(\frac{M}{\dt} +K + \alpha M  \right) C^{n+1}  &=  L^n_2 , \label{eq:lin-2}
\end{align}
where we have used the notation 
\[
    L^n_1 =  \f12\left( s_{1,h}^{\nn},u_h^n \right)  A^n S_1^{\nn} -r^n A^n S^n_1 + \frac{M}{\dt}  U^n + A^n C^n, \quad \text{and}\quad L^n_2 =  \frac{M}{\dt} C^n .
\]
Multiplying equation \eqref{eq:lin-1} by $M^{-1}$, we obtain 
\begin{equation}
    U^{n+1} + \f\dt2  \scal{s_{1,h}^{\nn}}{u_h^{n+1}}  M^{-1} A^n  S_1^{\nn} = \dt M^{-1} L_1^n.   
\end{equation}
Then, we take the inner product with $ S_1^{\nn}$ to obtain the linear system
\begin{equation}
    \scal{u_h^{n+1}}{s_{1,h}^{\nn}} + \f\dt2 \scal{u_h^{n+1}}{s_{1,h}^{\nn}} \left[S_1^n\right]^T M^{-1}  A^n  S_1^{\nn}= \dt \left[S_1^n\right]^T M^{-1}L^n_1.
\end{equation}
Thus, 
\begin{equation}   
    \scal{u_h^{n+1}}{s_{1,h}^{\nn}} = \dt \f{\left[S_1^n\right]^T M^{-1} L^n_1}{1+ \f\dt2 \left[S_1^n\right]^T M^{-1}  A^n  S_1^{\nn}}.
    \label{eq:obtain_scalar_prod}
\end{equation}
Then $U^{n+1}$ is obtained by inverting the constant mass matrix $M$ the equation \eqref{eq:lin-1}. As said before, for efficiency reasons the mass matrix can be replaced by the diagonal lumped matrix. Then, the solving of the equation requires only to invert a diagonal matrix.
$C^{n+1}$ is computed using the equation \eqref{eq:lin-2}: we just need to invert the constant M-matrix $\left(\frac{M}{\dt} +K + \alpha M  \right)$.

Hence, the solving of the problem reduces to the following computations:
\begin{enumerate}
    \item Compute $L_1^n$, $L^n_2$ and $S^n_1$ using the values from the previous time step.
    \item Solve the equation \eqref{eq:obtain_scalar_prod} to obtain $\scal{u_h^{n+1}}{ s^{\nn}_{1,h}}$.
    \item Solve the two equations~\eqref{eq:lin-1}--\eqref{eq:lin-2} to obtain $\{U^{n+1},C^{n+1}\}$.
\end{enumerate}

\section{Existence of a non-negative solution and stability bound}
\subsection{Existence of a discrete non-negative solution}
\begin{theorem}[Existence of a unique non-negative discrete solution]
    \label{th:existence}
        Let $d\le 3$ and assume that $\kappa_h>0$, $\Delta t >0$ such that
        \begin{equation}
            \f{\chi_c\, \kappa_h}{2 D_u} < 1.
            \label{eq:stabilite-peclet}
        \end{equation}
        Given an initial condition $\{u^0_h,c^0_h\}$ such that \eqref{eq:hyp-init-cond} and \eqref{eq:init-cond} are satisfied, there are two positive constants $C_1,C_2$ such that if
        \begin{equation}
            \f{C_1 \,\dt \, \chi_c}{\kappa_h} \le 1,
            \label{eq:pos-cond}
        \end{equation}
        and 
        \begin{equation}
            \f{ C_2  \, \dt\, D_u }{\kappa_h^2} \le 1,
            \label{eq:stable-diff}
        \end{equation} 
    then the problem~\eqref{eq:p_t u}--\eqref{eq:init-cond} admits a unique solution $\{u^{n+1}_h,c^{n+1}_h\}\in \fespace\times \fespace$ with 
    \[
        0\le u^{n+1}_h \le 1,\quad \text{ and }\quad 0\le c^{n+1}_h \le \ov c,
    \]  
    where $\ov c$ is a positive and finite constant.
\end{theorem}
\begin{proof}

    \noindent \textit{Step 1: Existence of a unique solution in $\fespace\times \fespace$.}
    As we have seen in the section describing the numerical scheme, the problem~\eqref{eq:p_t u}--\eqref{eq:init-cond} reduces to solving three linear equations. 
    To prove the existence of a unique pair of solutions $\{u^{n+1}_h,c^{n+1}_h \}$, we start by using equation \eqref{eq:mu_1} in \eqref{eq:p_t u} to write forall $ \phi  \in \fespace$
    \begin{equation}
        \lscal{u_h^{n+1}}{\bas}  = \lscal{u_h^n}{\bas} + \dt\left[ \parent{\f12\scal{u_h^n-u_h^{n+1}}{s^n_{1,h}} - r^n} \scal{ \vp(u^n_h) \nabla s^n_{1,h}}{\nabla \bas} + \scal{\vp(u^n_h) \nabla c_h^n}{\nabla \bas} \right].
        \label{eq:proof-ex-step-1}
    \end{equation}
    On the right-hand side, the only unknown comes from the term $\scal{u_h^n-u_h^{n+1}}{s^n_{1,h}}$. Let us show that it can be calculated from the solution of the previous time step. Using equation \eqref{eq:mu_1}, equation \eqref{eq:r_n_1} and replace $\mu_{1,h}^{n+1}$ in \eqref{eq:p_t u}, we obtain for all $\phi\in \fespace$
    \[
        \lscal{\f{u^{n+1}_h-u^n_h}{\dt}}{\bas} = \chi_c\left(\vp(u^n_h)\nabla c^n_h,\nabla \bas \right) - D_u \left( \f12 \lscal{s^n_{1,h}}{u^{n+1}_h-u^n_h}+r^n\right)\scal{\vp(u^n_h)\nabla s^n_{1,h}}{\nabla \bas}. 
    \] 
    Taking $\bas = s^n_{1,h}$ in the previous equation, we obtain the definition 
    \begin{equation}
        f(u^n_h,c^n_h) = \lscal{u^{n+1}_h-u^n_h}{s^{n}_{1,h}} = \f{\chi_c \,\dt \, \scal{\vp(u^n_h)\nabla c^n_h}{\nabla s^n_{1,h}} - D_u\,\dt \, r^n\, \int_\Omega \vp(u_h^n)\abs{\nabla s^n_{1,h}}^2\,\dd x }{1 + \f{D_u\, \dt}{2} \int_\Omega \vp(u_h^n)\abs{\nabla s^n_{1,h}}^2\,\dd x},   
        \label{eq:rn-explicit}
    \end{equation}
    and we know that $f$ is a continuous function of both of its arguments.
    Therefore, from the equation \eqref{eq:proof-ex-step-1}, we obtain for all $\phi\in \fespace$
    \[
        \lscal{u_h^{n+1}}{\bas} = \lscal{u_h^n}{\bas} + \dt\left[ -\parent{\f12 f(u^n_h,c^n_h) + r^n} \scal{ \vp(u^n_h) \nabla s^n_{1,h}}{\nabla \bas} + \scal{\vp(u^n_h) \nabla c_h^n}{\nabla \bas} \right].
    \]
    Consequently, the coefficients $u^{n+1}_i$, $ i =1,\dots,N_h$, are uniquely defined at each time step by the previous state of the solution.
    Then, the uniqueness of the solution $c^{n+1}_h$ follows the discrete version of the Lax-Milgram theorem.
    Altogether, we proved that it exists a unique solution $\{u^{n+1}_h,c^{n+1}_h\}\in \fespace\times \fespace$ of the problem.

    \noindent \textit{Step 2: Conservation of mass.}
    To prove mass conservation, we use the identity
        \begin{equation}
        \sum_{\substack{j \neq i \\ x_j \in T_i}} \abs{A_{ij}^n} = A_{ii}^n. 
        \label{eq:prop-A} 
    \end{equation}
    Therefore, for each $x_i\in J_h$, we have 
    \[
        \sum_{j=1}^{N_h}\scal{\bas_j}{\bas_i}\left(u^{n+1}_h-u^n_h\right)(x_j) = \dt\left[ -D_u \parent{\f12 f(u^n_h,c^n_h) + r^n} \sum_{j=1}^{N_h} A_{ij}^n s^n_{1,h}(x_j)+ \chi_c \sum_{j=1}^{N_h} A_{ij}^n c_h^n(x_j) \right]. 
    \]
    Summing over the nodes, we get
    \[
        \begin{aligned}
            \sum_{i=1}^{N_h}&\sum_{j=1}^{N_h}\scal{\bas_j}{\bas_i}\left(u^{n+1}_h-u^n_h\right)(x_j) \\
            &= \dt\left[ -D_u \parent{\f12 f(u^n_h,c^n_h) + r^n} \sum_{i=1}^{N_h}\sum_{j=1}^{N_h} A_{ij}^n s^n_{1,h}(x_j)+ \chi_c \sum_{i=1}^{N_h}\sum_{j=1}^{N_h} A_{ij}^n c_h^n(x_j) \right]. 
        \end{aligned}
    \]
    Using the symmetry of the matrix A, the property \eqref{eq:prop-A} and the fact that the mesh is acute, we obtain 
    \[
        \sum_{i=1}^{N_h}\sum_{j=1}^{N_h}\scal{\bas_j}{\bas_i}\left(u^{n+1}_h-u^n_h\right)(x_j) = 0,
    \]
    which implies mass conservation .

    \noindent \textit{Step 3: Non-negativity and $L^\infty$ bound for $\{u_h^{n+1},c^{n+1}_h\}$.}
    Using the equation \eqref{eq:proof-ex-step-1}, we find for all $\bas\in\fespace$
    \begin{equation}
        \lscal{u_h^{n+1}}{\bas}  = \lscal{u_h^n}{\bas} + \dt\left[ -\tilde D \scal{\nabla u^n_h}{\nabla \bas} + \scal{\vp(u^n_h) \nabla c_h^n}{\nabla \bas} \right],
        \label{eq:u-reconstitue}
    \end{equation}
    where the diffusion coefficient is given by 
    \begin{equation}
        \begin{aligned}
            -\tilde D &= -D_u \f{r^{n+1}}{\sqrt{\cae_1[u^n_h]}}\\
            &=D_u\left(\f{\lscal{u_h^n-u_h^{n+1}}{s^n_{1,h}} - 2 r^n}{2\sqrt{\cae_1[u^n_h]}}\right) , \\
            &= \f{D_u}{2\cae_1[u^n_h]} \left[\lscal{P_h\left(\f{\p \cae_1[u^n_h]}{\p u^n_h}\right)}{u^n_h-u^{n+1}_h} - 2 r^n \sqrt{\cae_1[u^n_h]} \right],\\
        \end{aligned}
        \label{eq:ineq-diffusion-coefficient}
    \end{equation}
    Therefore, for each node $x_i \in J_h$, we have
    \begin{equation}
        \begin{aligned}
        \sum_{j=1}^{N_h} u_h^{n+1}(x_j) M_{ij} &= \sum_{j=1}^{N_h} u^n_h(x_j)M_{ij} \\
        &+ \dt \sum_{x_j \in \Lambda_i} \left[ \chi_c A_{ij}^n\left(c^n_h(x_j)-c^n_h(x_i)\right)  - \f{D_u r^{n+1}}{\sqrt{\cae_1[u^n_h]}} K_{ij}\left(u^n_h(x_j)-u^n_h(x_i) \right) \right],
        \end{aligned}
    \end{equation}    
    Because the spatio-temporal mesh satisfies the conditions \eqref{eq:stabilite-peclet} and 
    \begin{equation}
        \f{\dt\, \chi_c\, G_h \norm{b^n}_\infty}{\kappa_h} \le 1,
        \label{eq:firstcond-proof}
    \end{equation}
    where 
    \[
        \norm{b^n}_\infty = \sup_{\substack{ i=1,\dots,N_h \\
        j\in \Lambda_i}} A_{ij}^n\abs{c^n_j - c^n_i}, 
    \]
    and
    \begin{equation}
        \f{ \dt\, D_u \, G_h \, r^{n+1} }{\kappa_h^2 \sqrt{\cae_1[u^n_h]}} \le 1,
    \end{equation}
    one can easily find that $0\le u^{n+1}_h\le 1$. The only difficulty with this latter is because it depends on $r^{n+1}$. However, as seen in equation \eqref{eq:rn-explicit}, this term can be calculated from the solution $\{u^n_h,c^n_h\}$. Furthermore, this coefficient is bounded at all time $t^n$ since we assumed that $r^n$ is bounded and $u^n_h \in [0,1]$. 
    For the condition \eqref{eq:firstcond-proof}, since $c^n_h$ is bounded and in $\fespace$, $\norm{b^n}_\infty$ remains bounded from above at all time. 
    Therefore, there are two positive constants $C_1$ and $C_2$ such that
    \[
        C_1 \ge G_h \norm{b^n}_\infty \quad\text{and}\quad C_2 \ge \f{ G_h  r^{n}}{\sqrt{\cae_1[u^{n-1}_h]}}  \quad\text{ for } n=1,...,N_T,
    \]
    and we obtain the two conditions \eqref{eq:pos-cond}, \eqref{eq:stable-diff}.
    Under these time-dependent assumptions, we proved that the numerical scheme preserves the bounds 
    \[
        0\le u^{n+1}_h \le 1.  
    \]
    From this result, the non-negativity and the existence of an upper bound $\ov c$ such that
    \[
        0\le c^{n+1}_h \le \ov c,  
    \]
    is trivially found from the properties of M-matrices. 
    This finishes the proof of the existence of the solution of the problem ~\eqref{eq:p_t u}--\eqref{eq:init-cond}.
\end{proof}

\subsection{Discrete energy a priori estimate}
Since we are using the SAV method, we are preserving the energy at the discrete level.
\begin{proposition}[Discrete energy]
    \label{prop:energy}
    Consider a solution $\{u^{n+1}_h,c^{n+1}_h\}$ defined by Theorem \ref{th:existence},
     the discrete energy of the system ~\eqref{eq:p_t u}--\eqref{eq:r_n_1} is given by
    \begin{equation}
        E(u^{n+1}_h,c^{n+1}_h) = \f12\parent{\abs{c^{n+1}_h}^2_1 + \alpha  \norm{c^{n+1}_h}^2_{0} }   +  B\abs{r^{n+1}}^2  -\lscal{c^{n+1}_h}{u^{n+1}_h},
    \label{eq:discrete-energy}
    \end{equation}
    and 
    \begin{equation}
        \f{\dd E}{\dd t } := \f{E^{n+1}-E^n}{\dt} = -\left( \norm{\mu_{2,h}^{n+1}}^2_0 + \int_\Omega \vp(u^n_h)\abs{\nabla \mu_{1,h}^{n+1}}^2 \, \dd x \right).
    \label{eq:deriv-discr-energy}
    \end{equation}
\end{proposition}
\begin{proof}
    Starting from equation \eqref{eq:p_t u} with $\bas = \mu_{1,h}^{n+1}$, we have
    \[
     \lscal{u^{n+1}_h-u^n_h}{\mu_{1,h}^{n+1}}=  -\dt  \int_\Omega \vp(u^n_h)\abs{\nabla \mu_{1,h}^{n+1}}^2\,\dd x.  
    \] 
    The same can be done starting from equation \eqref{eq:p_t c} to obtain
    \[
        \lscal{c^{n+1}_h-c^n_h}{\mu_{2,h}^{n+1}} = -\dt \norm{\mu^{n+1}_{2,h}}_0^2.  
    \]
    Therefore, summing the two previous equations, we obtain 
    \[
        \lscal{u^{n+1}_h-u^h_n}{\mu_{1,h}^{n+1}} + \lscal{c^{n+1}_h-c^h_n}{\mu_{2,h}^{n+1}} = - \Delta t\left( \norm{\mu_{2,h}^{n+1}}^2_0 + \int_\Omega \vp(u^n_h)\abs{\nabla \mu_{1,h}^{n+1}}^2 \, \dd x \right),
    \]
    from which we conclude \eqref{eq:deriv-discr-energy}. Consequently, we already recover the monotonic decay of the discrete energy. To obtain the expression of the energy, we replace $\bas = u^{n+1}_h-u^n_h$ in \eqref{eq:mu_1} to get
    \[
        \lscal{u^{n+1}_h-u^n_h}{\mu_{1,h}^{n+1}} = -\lscal{u^{n+1}_h-u^n_h}{c^n_h} + B r^{n+1}\lscal{u^{n+1}_h-u^n_h}{s^n_{1,h}}. 
    \]
    However, using the equation \eqref{eq:r_n_1}, we have
    \[
        \lscal{u^{n+1}_h-u^n_h}{\mu_{1,h}^{n+1}} = -\lscal{u^{n+1}_h-u^n_h}{c^n_h} + 2 B r^{n+1}\parent{r^{n+1}-r^n}.
    \]
    Moreover, using the inequality $a(a-b) \ge \f12\left(a^2-b^2 \right)$, we get
    \begin{equation}
        \lscal{u^{n+1}_h-u^h_n}{\mu_{1,h}^{n+1}} \ge -\lscal{c^n_h}{u^{n+1}_h-u^h_n}  +  B\abs{r^{n+1}_{1}}^2- B\abs{r^{n}_{1}}^2.
        \label{eq:energy-proof1}
    \end{equation}
    Then, performing the same calculations starting from the equation \eqref{eq:mu_2}, we obtain 
    \begin{equation}
        \scal{c^{n+1}_h-c^h_n}{\mu_{2,h}^{n+1}} \ge \f12 \left[\abs{c^{n+1}_h}^2_1 - \abs{c^n_h}^2_1 + \alpha \left(\norm{c^{n+1}_h}^2_0-\norm{c^n_h}^2_0 \right)\right] - \scal{u^{n+1}_h}{c^{n+1}_h-c^n_h}.
        \label{eq:energy-proof2}
    \end{equation}
    Summing equation \eqref{eq:energy-proof1} with \eqref{eq:energy-proof2}, we obtain the inequality
    \[
        \begin{aligned}
         \f12 \left[\abs{c^{n+1}_h}^2_1 - \abs{c^n_h}^2_1 +\alpha \left(\norm{c^{n+1}_h}^2_0 -\norm{c^n_h}^2_0\right) \right]+B \abs{r^{n+1}_{1}}^2- B\abs{r^{n}_{1}}^2- \lscal{u^{n+1}_h}{ c^{n+1}_h} + \lscal{u^{n}_h}{c^{n}_h}  \\
        \le   - \Delta t\left( \norm{\mu_{2,h}^{n+1}}^2_0 + \int_\Omega \vp(u^n_h)\abs{\nabla \mu_{1,h}^{n+1}}^2 \, \dd x \right),
        \end{aligned}
    \]
    from which we deduce the definition and the decay of the discrete energy ~\eqref{eq:discrete-energy}--\eqref{eq:deriv-discr-energy}.

\end{proof}
\begin{remark}
From the fact that both $u^{n+1}_h$ and $c^{n+1}_h$ are bounded (see theorem \ref{th:existence}), the energy defined by \eqref{eq:discrete-energy} is bounded from below and can be used to obtain inequalities. 
\end{remark}
\section{Convergence analysis}
\subsection{Notations}
To prove the convergence of the discrete solutions, we need some further notations. 
We define the sequence of approximate solutions 
\[
    u_{h\dt} = \left(u^0_h, \dots, u^{N_T}_h\right) \text{ and } c_{h\dt} = \left(c^0_h, \dots, c^{N_T}_h\right),
\]
and each of them lies in the Cartesian product space $V^{N_T+1}_h$. 
To construct the sequences $u_{h\dt}, c_{h\dt}$, we define the linear operator $S_{i h\dt}: \fespace \times \fespace  \to \fespace^{N_T+1}$ where $i=1,2$ and we have  
\[
    \begin{aligned}
    S_{1 h\Delta t}(u_{0h},c_{0h}) &= (u_{h}^0,\dots,u^{N_T}_h)=u_{h\dt},\\
    S_{2 h\Delta t}(u_{0h},c_{0h}) &= (c_{h}^0,\dots,c^{N_T}_h)=c_{h\dt}.
    \end{aligned}
\]
These two operators are inductively defined by the system
\begin{align}
        \lscal{\f{u^{n+1}_h-u^n_h}{\dt}}{\bas} &= \parent{\f12\lscal{u_h^n-u_h^{n+1}}{s^n_{1,h}} - r^n_1} \scal{ \nabla u^n_{h}}{\nabla \bas} + \scal{\vp(u^n_h) \nabla c_h^n}{\nabla \bas}, \label{eq:conv-pb-1}\\
        \lscal{\f{c^{n+1}_h-c^n_h}{\dt}}{\bas} &= -\scal{\nabla c^{n+1}_h}{\nabla \bas} -\alpha\lscal{c^{n+1}_h}{\bas} + \lscal{u^{n+1}_h}{\bas}.\label{eq:conv-pb-2}
\end{align}

\subsection{Preliminary results}
Let us define the quantity
\begin{equation}
    c_i(h) = \max_{v_h \in S^h}\f{\norm{v_h}_{H^1(\Omega)}}{\norm{v_h}_{L^2(\Omega)}}.
    \label{eq:inv-ineq}
\end{equation}
\begin{proposition}[Inverse inequality]
    Assuming that the mesh is quasi-uniform, the quantity \eqref{eq:inv-ineq} is finite and we have 
    \[
        c_i(h) \le C h^{-1},  
    \]
    where $C$ is a positive constant.
\end{proposition}
\begin{proof}
    The proof is given in \cite{ern_2004_fe}, corollary $1.141$ on global inverse inequalities. 
\end{proof}
\subsection{Stability of the scheme}
\begin{proposition}[Stability]
    Let the spatio-temporal mesh satisfy the inequalities \eqref{eq:pos-cond} and \eqref{eq:stable-diff}.
    Let $\{u^{n+1}_h,c^{n+1}_h\}$ be the solution of the discrete problem~\eqref{eq:p_t u}--\eqref{eq:r_n_1} that is defined by Theorem \ref{th:existence}. The following inequalities hold
    \begin{equation}
        \norm{u^{n+1}_h}^2_0  + \dt^2  \sum_{n=0}^{N_T-1} \norm{\f{u^{n+1}_h-u^n_h}{\dt}}^2_0 + \dt  \sum_{n=0}^{N_T-1} \norm{u^{n+1}_h}_1^2  \le C +\norm{u^0_h}^2_0 , \label{eq:ineq-conv-11}
    \end{equation}
    and 
    \begin{equation}
        \norm{c^{n+1}_h}^2_{0} + \dt\sum_{n=0}^{N_T-1} \norm{\f{c^{n+1}_h-c^n_h}{\dt}}^2_{0} + \dt \sum_{n=0}^{N_T-1} \norm{c^{n+1}_h}_1 \le  C +\left(\norm{u^0_h}^2_0 + \norm{c^0_h}^2_0\right),\label{eq:ineq-conv-21}
    \end{equation}
\end{proposition}
\begin{proof}
    
    \noindent \textit{Proof of the inequality \eqref{eq:ineq-conv-11}.}
    Starting from equation \eqref{eq:conv-pb-1}, taking $\bas = 2\dt u^{n+1}_h$ and using the property $2a(a-b) = a^2-b^2+(a-b)^2$, we have
    \[
        \norm{u^{n+1}_h}^2_{0} + \norm{u^{n+1}_h-u^n_h}^2_{0} - \norm{u^n_h}^2_{0} = 2\dt \chi_c \scal{\vp(u^n_h)\nabla c^n_h}{\nabla u^{n+1}_h} -2\dt \f{r^{n+1}}{\sqrt{\cae_1[u^n_h]}} \scal{\nabla u^n_h}{\nabla u^{n+1}_h}. 
    \]
    However, using the coercivity of the operator $a(t^{n+1},\cdot,\cdot)= \scal{\nabla \cdot}{\nabla \cdot}$, there is a positive constant $\alpha$ such that 
    \[
        \begin{aligned}
        \scal{\nabla u^n_h}{\nabla u^{n+1}_h} &= \scal{\nabla u^{n+1}_h}{\nabla u^{n+1}_h} - \scal{\nabla u^{n+1}_h - \nabla u^n_h}{\nabla u^{n+1}_h}\\
        &\ge \alpha \norm{u_h^{n+1}}_1^2 - \norm{u^{n+1}_h-u^n_h}_1\norm{u^{n+1}_h}_1.  
        \end{aligned}
    \]
    From the inverse inequality \eqref{eq:inv-ineq} and Young's inequality, there is $0<\kappa_1 <1$ such that
    \[
        \begin{aligned}
            \scal{\nabla u^n_h}{\nabla u^{n+1}_h} &\ge \alpha \norm{u_h^{n+1}}_1^2 - c_i(h) \norm{u^{n+1}_h-u^n_h}_0\norm{u^{n+1}_h}_1\\
            &\ge \alpha\left(1-\f{\kappa_1}{2}\right) \norm{u_h^{n+1}}_1^2 - \f{c_i(h)^2}{2\kappa_1 \alpha} \norm{u^{n+1}_h-u^n_h}^2_0.
        \end{aligned}
    \]
    Similarly, there is a constant $0<\kappa_2<1$ such that
    \[
        \begin{aligned}
        \chi_c\scal{\vp(u^n_h)\nabla c^n_h}{\nabla u^{n+1}_h} &\le  \chi_c\norm{\vp}_\infty \norm{c^n_h}_1 \norm{u^{n+1}_h}_1\\
        &\le \f{\chi_c^2}{2 \kappa_2} \norm{c^n_h}_1^2 + \f{\norm{\vp}^2_\infty \kappa_2}{2}\norm{u^{n+1}_h}_1^2.
        \end{aligned} 
    \]
    Altogether, we obtain the inequality
    \begin{equation}
        \begin{aligned}
            \norm{u^{n+1}_h}^2_{0}&+\left(1-\f{\dt (c_i(h))^2 D_u r^{n+1}}{\alpha \kappa_1 \sqrt{\cae_1[u^n_h]}}\right) \norm{u^{n+1}_h-u^n_h}^2_{0} - \norm{u^n_h}^2_{0}  \\
            &+ \dt\left(\f{ D_u r^{n+1}\alpha(2-\kappa_1)}{\sqrt{\cae_1[u^n_h]}}-\kappa_2\norm{\vp}_\infty^2\right)\norm{u^{n+1}_h}^2_1 \le \dt \f{\chi_c^2}{\kappa_2}\norm{c^n_h}_1^2.  
        \end{aligned}
        \label{eq:main-ineq-proof-11}
    \end{equation}
    With the choice $\vp(s)=s(1-s)$, we assume that $\kappa_1$,$\kappa_2$ and $\dt$ satisfy
    \begin{equation}
        \kappa_2 \le \f{4r^{n+1}\alpha(2-\kappa_1)}{\sqrt{\cae_1[u^n_h]}} \quad \text{and} \quad \dt \le \f{(c_i(h))^2 D_u r^{n+1}}{\alpha \kappa_1 \sqrt{\cae_1[u^n_h]}},
        \label{eq:cond-stab}
    \end{equation}
    where the second condition is strongly related to \eqref{eq:stable-diff}. Hence, using the previous assumptions together with $c_p\norm{u^{n+1}_h}_0 \le \norm{u^{n+1}_h}_1$, it exists a positive constant $C$ such that 
    \[
        (1+\dt c_p \, C)\norm{u^{n+1}_h}^2_{0,h}- \norm{u^n_h}^2_{0}  +\left(1-\f{\dt (c_i(h))^2 D_u r^{n+1}}{\alpha \kappa_1 \sqrt{\cae_1[u^n_h]}}\right) \norm{u^{n+1}_h-u^n_h}^2_{0} \le \dt \f{\chi_c^2}{\kappa_2}\norm{c^n_h}_1^2,
    \]
    and summing the previous inequality from $n=0\to N_T-1$, we have
    \[
        \norm{u^{n+1}_h}^2_{0}-\norm{u^0_h}^2_{0}  + \sum_{n=0}^{N_T-1}C \dt^2 \norm{\f{u^{n+1}_h-u^n_h}{\dt}}^2_{0} \le T\f{\chi_c^2}{\kappa_2} \max_{n=0,\dots,N_T-1} \norm{c^n_h}_1^2 \le C_1.
    \]
    The right-hand side of the previous inequality is bounded by a constant that we denoted $C_1$ due to the energy inequality.
    Then, since we assumed that the conditions \eqref{eq:cond-stab} holds, it exists a positive constant $C_2$ such that summing the equation \eqref{eq:main-ineq-proof-11} from $n=0\to N_T-1$ gives
    \[
        C_2 \sum_{n=0}^{N_T-1} \norm{u^{n+1}_h}_1^2 \le \f{\chi_c^2}{\kappa_2}\sum_{n=0}^{N_T-1} \norm{c^{n}_h}_1^2,
    \]
    where the right-hand side is bounded using the energy estimate ~\eqref{eq:discrete-energy}--\eqref{eq:deriv-discr-energy} and we obtain \eqref{eq:ineq-conv-11}. \par 
    
    \noindent \textit{Proof of the inequality \eqref{eq:ineq-conv-21}.}
    Starting from the equation \eqref{eq:conv-pb-2} and taking $\bas = 2\dt c^{n+1}_h$, we have
    \[
        \norm{c^{n+1}_h}^2_{0} + \norm{c^{n+1}_h-c^n_h}^2_{0} - \norm{c^n_h}^2_{0} + 2\dt b(t^{n+1},c^{n+1}_h,c^{n+1}_h) = \lscal{u^{n+1}_h}{c^{n+1}_h},
    \]  
    where $b(t^{n+1},c^{n+1}_h,\bas) = \scal{\nabla c^{n+1}_h}{\nabla \bas} + \alpha \lscal{c^{n+1}_h}{\bas}$. Furthermore, we know that it exists a positive real value $\alpha_2$ such that  
    \[
        b(t^{n+1},c^{n+1}_h,c^{n+1}_h) \ge \abs{c^{n+1}_h}_1^2 + \alpha\norm{c^{n+1}_h}^2_0 \ge \alpha_2 \norm{c^{n+1}_h}^2_1.
    \]
    Therefore, again using Young's inequality, for $0<\kappa_3<1$, we have 
    \begin{equation}
        (1-\dt\kappa_3)\norm{c^{n+1}_h}^2_{0} + \norm{c^{n+1}_h-c^n_h}^2_{0} -\norm{c^n_h}^2_{0} + 2\dt \alpha_2 \norm{c^{n+1}_h}^2_1 \le \f{\dt}{\kappa_3}\norm{u^{n+1}_h}^2_{0}.
        \label{eq:ineq-proof-stab-2}
    \end{equation}
    From the inequality $c_p \norm{c^{n+1}_h}_0 \le \norm{c^{n+1}_h}_1$, we obtain 
    \[
    \left(1+\dt\left( 2\alpha_2 c_p -\kappa_3 \right)\right)\norm{c^{n+1}_h}^2_{0} - \norm{c^n_h}^2_{0}  + \norm{c^{n+1}_h-c^n_h}^2_{0} \le  \f{\dt}{\kappa_3}\norm{u^{n+1}_h}^2_0.
    \]
    We assume that the condition 
    \[
        0\le2\alpha_2 c_p -\kappa_3,
    \]
    is satisfied.
    Hence, summing from $n=0\to N_T-1$ and using \eqref{eq:ineq-conv-11}, we obtain 
    \[
        \norm{c^{n+1}_h}^2_{0} + \dt^2\sum_{n=0}^{N_T-1} \norm{\f{c^{n+1}_h-c^n_h}{\dt^2}}^2_{0} \le (d+2)\left(\norm{c^0_h}^2_{0} + \f{\dt}{\kappa_3} \sum_{n=0}^{N_T-1} \norm{u^{n+1}_h}^2_{0}\right).
    \]
    Moreover, from the equation \eqref{eq:ineq-proof-stab-2}, we know that 
    \[
        2\dt \alpha_2 \norm{c^{n+1}_h}^2_1 \le \f{\dt}{\kappa_3}\norm{u^{n+1}_h}^2_0.  
    \]
    Summing the previous inequality from $n=0\to N_T-1$, we have 
    \[
         \sum_{n=0}^{N_T-1} \norm{c^{n+1}_h}_1^2 \le \f{1}{2 \kappa_3}\sum_{n=0}^{N_T-1}\norm{u^{n+1}_h}^2_0,  
    \] 
    which gives \eqref{eq:ineq-conv-21}.
\end{proof}

\subsection{Convergence}
To study the convergence of the scheme, we introduce the following notations for $n = 0, \dots, N_T - 1$
\begin{equation*}
	U_h(t,x) := \frac{t-t^n}{\Delta t} u_h^{n+1} + \frac{t^{n+1}-t}{\Delta t} u_h^{n}, \quad t \in (t^n, t^{n+1}],
\end{equation*}
and
\begin{equation*}
	\frac{\partial U_h}{\partial t} := \frac{u_h^{n+1}-u_h^n}{\Delta t} \quad t \in (t^n, t^{n+1}].
\end{equation*}
We also define
\[
    U^+_h :=  u_h^{n+1},\quad U^-_h := u_h^n,
\]
and
\begin{equation*}
	U_h - U_h^+ = (t-t^{n+1}) \frac{\partial U_h}{\partial t},\quad U_h - U_h^- = (t-t^{n}) \frac{\partial U_h}{\partial t} \quad t \in (t^n, t^{n+1}], \quad n\ge 0.
\end{equation*}

We also have the analogous definitions for $C_h$ which are for $n = 0, \dots, N_T - 1$
\begin{equation*}
	C_h(t,x) := \frac{t-t^n}{\Delta t} c_h^{n+1} + \frac{t^{n+1}-t}{\Delta t} c_h^{n}, \quad t \in (t^n, t^{n+1}],
\end{equation*}
\begin{equation*}
	\frac{\partial C_h}{\partial t} := \frac{c_h^{n+1}-c_h^n}{\Delta t} \quad t \in (t^n, t^{n+1}], 
\end{equation*}
\[
    C^+_h :=  c_h^{n+1},\quad C^-_h := c_h^n,
\]
and
\begin{equation*}
	C_h - C_h^+ = (t-t^{n+1}) \frac{\partial C_h}{\partial t}, \quad \text{and} \quad C_h - C_h^{-} = (t-t^{n}) \frac{\partial C_h}{\partial t} \quad t \in (t^n, t^{n+1}], \quad n\ge 0.
\end{equation*}
We also define the pair of function $\{u,c\}$ such that 
\begin{equation}
\begin{cases}
    &u\in   L^\infty\left(0,T;H^1\left(\Omega\right) \right) \bigcap H^1\left([0,T];\left(H^1\left(\Omega\right)\right)^\prime\right)\bigcap L^2\left([0,T];L^2(\Omega)\right), \\
    &c\in L^\infty\left(0,T;H^1\left(\Omega\right) \right) \bigcap H^1\left([0,T];\left(H^1\left(\Omega\right)\right)^\prime\right)\bigcap L^2\left([0,T];L^2(\Omega)\right),\\
    &0\le u \le 1,\quad 0\le c \le \ov c, \quad \text{a.e. in } \Omega_T,
    \label{eq:def-pair-fun-1}
\end{cases}    
\end{equation}
where $\ov c$ is a finite constant that depends on $\alpha$.

\begin{theorem}[Convergence]
    Let $d=1,2,3$ and $\{u^0,c^0\} \in H^1(\Omega)\times H^1(\Omega),$ with $0\le u^0 <1$ a.e. $\Omega$. We assume that $\{ \mathcal{T}^h, u^0_h, c^0_h,\Delta t\}_{h>0}$ satisfy
    \begin{enumerate}
        \item $\{u^0_h, c^0_h\} \in \fespace \times \fespace$ given by \eqref{eq:init-cond}.
        \item Let $\Omega \subset \R^d$ be a polyhedral domain and $\mathcal{T}^h$ an acute mesh of it into $N$ mesh elements. 
    \end{enumerate}
    Therefore, for $\Delta t, h \to 0$, it exists a subsequence of solutions $\{U_h,C_h\}$ and a pair of function $\{u,c\}$ defined by \eqref{eq:def-pair-fun-1} such that 
    \begin{align}
        U_h &\to u,\quad \text{strongly in } L^2\left(0,T;L^2\left(\Omega\right) \right), \label{eq:conv-strong-u}\\
        U_h &\rightharpoonup u,\quad \text{weakly in } L^2\left(0,T;H^1\left(\Omega\right) \right), \label{eq:conv-weak-u} \\
        \frac{\partial U_h}{\partial t} &\rightharpoonup \f{\p u}{\p t}, \quad\text{weakly in } L^2\left(0,T;\left(L^2\left(\Omega\right)\right)' \right), \label{eq:conv-weak-dt-u}\\
        C_h &\to c,\quad \text{strongly in } L^2\left(0,T;L^2\left(\Omega\right) \right), \label{eq:conv-strong-c}\\
        C_h &\rightharpoonup c,\quad \text{weakly in } L^2\left(0,T;H^1\left(\Omega\right) \right), \label{eq:conv-weak-c} \\
        \frac{\partial C_h}{\partial t} &\rightharpoonup \f{\p c}{\p t}, \quad\text{weakly in } L^2\left(0,T;\left(L^2\left(\Omega\right)\right)' \right), \label{eq:conv-weak-dt-c}\\
        r^{n+1} &\rightharpoonup r(t) = \sqrt{\cae_1(u(t))} \quad\text{weak-star in } L^\infty\left( 0,T\right).  \label{eq:conv-weak-r}
    \end{align}
    Moreover, for all $\eta \in L^2\left([0,T];H^1(\Omega) \right)$, $\{u,c\}$ is a solution of the limit model
    \begin{equation}
    \begin{cases}
        \int_0^T \left< \f{\p u}{\p t}, \eta\right> \,\dd t &= \chi_c \int_0^T \int_\Omega \vp(u) \nabla c \nabla \eta \,\dd x\,\dd t-D_u \int_0^T \int_\Omega \nabla u\nabla \eta \,\dd x \,\dd t, \\
        \int_0^T \left< \f{\p c}{\p t}, \eta\right> \,\dd t &= -\int_0^T \int_\Omega \left[{\nabla c}{\nabla \eta} + \alpha {c}{\eta} - u\eta \right]\,\dd x\,\dd t.
    \end{cases}
    \label{eq:limit-model}
    \end{equation}
    which is the weak form of the SAV Keller-Segel model~\eqref{eq:KS-1}--\eqref{eq:KS-2}.
\end{theorem}
\begin{proof}
    
    \noindent\textit{Step 1: Weak and strong convergences.} The weak convergences \eqref{eq:conv-weak-u}, \eqref{eq:conv-weak-dt-u} are obtained from the inequality \eqref{eq:ineq-conv-11}. Then, from the compact embedding $H^1(\Omega) \subset L^2(\Omega) \equiv (L^2(\Omega))'$, we can apply the Lions-Aubin lemma to prove the strong convergence \eqref{eq:conv-strong-u}. The same can be applied for the weak convergences \eqref{eq:conv-weak-c}, \eqref{eq:conv-weak-dt-c} obtained from \eqref{eq:ineq-conv-21} and the strong convergence \eqref{eq:conv-strong-c}. The weak-star convergence 
    \[
        r^{n+1} \rightharpoonup r \quad \text{weak-star in} \quad L^\infty(0,T),   
    \]
    is given by the energy estimate ~\eqref{eq:deriv-discr-energy}--\eqref{eq:discrete-energy}.\par 
    \noindent \textit{Step 2: Limit equation.} Let us use $\bas = \pi^h \eta$ in \eqref{eq:conv-pb-1} where $\eta\in H^1\left(0,T;H^1(\Omega)\right)$ and analyze the convergence of the resulting terms separately. First, using the strong convergence \eqref{eq:conv-strong-u}, the weak convergence \eqref{eq:conv-weak-c}, the uniform convergence for the interpolation \eqref{eq:ineq-lumped-scalar-product2} and Lebesgue's dominated convergence theorem, we have 
    \[
        \chi_c \int_{0}^{T}\scal{\vp(U^-_h)\nabla C^-_h}{\nabla \pi^h(\eta) } \,\dd t\to \chi_c \int_0^T \int_\Omega \vp(u) \nabla c \nabla \eta \,\dd x\,\dd t.   
    \]
    Next, we want to prove that 
    \[
        \cae_1[U^-_h] \rightharpoonup \cae_1[u]\quad \text{weak-star in} \quad L^\infty\left(0,T\right).  
    \]
    From theorem \ref{th:existence}, we know that $\norm{U_h^-}_{L^\infty\left([0,T]\times \Omega\right)}\le C$. Therefore, it exists a positive constant $L$ such that  
    \begin{equation}
        \abs{\cae_1\left[U^-_h\right]-\cae_1[u]} \le  L \abs{U^-_h  - u}.
        \label{eq:conv-cae}
    \end{equation}
    Hence, from the weak convergences \eqref{eq:conv-weak-u}, \eqref{eq:conv-weak-r},\eqref{eq:conv-cae} and \eqref{eq:ineq-lumped-scalar-product2}, we have 
    \[
       D_u \int_0^T \f{r^{n+1}}{\sqrt{\cae_1\left[U^-_h\right]}}\scal{\nabla U^-_h}{\nabla \pi^h\eta} \,\dd t \to D_u \int_0^T \f{r}{\sqrt{\cae_1[u]}}\int_\Omega \nabla u\nabla \eta \,\dd x \,\dd t = D_u \int_0^T \int_\Omega \nabla u\nabla \eta \,\dd x \,\dd t.
    \]
    Then, for any $\eta \in H^1([0,T];H^1(\Omega))$, by integration by parts we have  
    \[
        \int_0^T\lscal{\f{\p U_h}{\p t}}{\pi^h \nu}\,\dd t = -\int_0^T \lscal{U_h}{\f{\p \left(\pi^h \eta\right)}{\p t}}\,\dd t + \lscal{U_h(T)}{\pi^h\eta(T)} - \lscal{U_h(0)}{\pi^h\eta(0)}.  
    \] 
    Hence, from the regularity of $\eta$, \eqref{eq:ineq-lumped-scalar-product2} and the convergence \eqref{eq:conv-strong-u} , we obtain
    \[
        \int_0^T \lscal{U_h}{\f{\p \left(\pi^h \eta\right)}{\p t}}\,\dd t  \to \int_0^T \scal{u}{\f{\p \eta}{\p t}}\,\dd t\quad \text{as}\quad h\to 0\quad\text{and }\quad \forall \eta  \in H^1([0,T];H^1(\Omega)).
    \]
    Combining the previous results, we get 
    \[
        \begin{aligned}
        \scal{u(T)}{\eta(T)} &- \scal{u(0)}{\eta(0)}-\int_0^T \scal{u}{\f{\p \eta}{\p t}}\,\dd t = \\
        &\chi_c \int_0^T \int_\Omega \vp(u) \nabla c \nabla \eta \,\dd x\,\dd t-D_u \int_0^T \f{r}{\sqrt{\cae_1[u]}}\int_\Omega \nabla u\nabla \eta \,\dd x \,\dd t.
        \end{aligned}
    \]
    From the energy estimate \eqref{eq:deriv-discr-energy}, we know that $\phi(u)\nabla c-D_u\f{r}{\sqrt{\cae_1[u]}}\nabla u$ is  in $L^2(\Omega_T)$. Hence, we know that $u\in H^1\left([0,T];\left(H^1(\Omega)\right)^\prime \right)$ and we obtain the first equation of the system \eqref{eq:limit-model}.
    Then taking $\eta \in H^1\left([0,T];H^1(\Omega) \right)$ in \eqref{eq:conv-pb-2}, passing to the limit $\dt,h \to 0$ in the right-hand side is performed using the convergence \eqref{eq:conv-weak-c} and \eqref{eq:conv-weak-u}. Altogether, we obtain the limit model \eqref{eq:limit-model}.
\end{proof}


\section{Conclusion}
We presented the application of scalar auxiliary variable method to the parabolic-parabolic Keller-Segel with volume filling using the gradient flow structure of the model. The resulting equations were approximated using a simple $P-1$ finite element method. The system is composed of two linear decoupled equations that can be solved efficiently. We were able to prove for this system the existence of a unique non-negative solution and the preservation of the monotonic decay of the discrete energy. We must stress that from the use of the SAV method, the energy that we are able to recover is a  modified version of the standard one. However, we were to prove that, in the limit of the discretization parameters, subsequences of solutions converge in Bochner spaces and the limit is the solution of the weak form of the Keller-Segel model. 

%
%
%


\bibliographystyle{siam}  \bibliography{biblio}

\end{document}